\documentclass[a4paper,12pt]{amsart}
\usepackage{amsmath}
\usepackage{amssymb}
\usepackage{amsthm,amsxtra}
\usepackage{cite}
\usepackage{enumerate}
\usepackage{enumitem}
\usepackage[mathscr]{eucal}
\usepackage{adjustbox}
\usepackage{float}
\usepackage{graphicx}
\usepackage{mathtools}
\usepackage{tikz-cd}
\usepackage{tkz-graph}
\usepackage{tkz-berge}
\usetikzlibrary{positioning,arrows,shapes.geometric,trees,decorations.text}
\usepackage{geometry}
\newgeometry{vmargin={40mm}, hmargin={30mm,30mm}}
\usepackage{cleveref}

\newtheorem{Def}{Definition}[section]
\newtheorem{Th}[Def]{Theorem}
\newtheorem{Ex}[Def]{Example}
\newtheorem{Lemma}[Def]{Lemma}
\newtheorem{Prop}[Def]{Proposition}
\newtheorem{Cor}[Def]{Corollary}
\newtheorem{Rem}[Def]{Remark}

{}
\newtheorem{Prob}[Def]{Problem}

\DeclareMathOperator{\fin}{Fin}

\DeclareMathOperator{\iwfs}{\mathcal{I}WFS}
\DeclareMathOperator{\iwfsd}{\mathcal{I}WFSD}

\begin{document}
\title[On a variation of selective separability using ideals]{On a variation of selective separability using ideals}

\author[ D. Chandra, N. Alam and D. Roy ]{ Debraj Chandra$^*$, Nur Alam$^\dag$ and Dipika Roy$^*$ }
\newcommand{\acr}{\newline\indent}
\address{\llap{*\,}Department of Mathematics, University of Gour Banga, Malda-732103, West Bengal, India}
\email{debrajchandra1986@gmail.com, roydipika1993@gmail.com}
\address{\llap{\dag\,}Department of Mathematics, Directorate of Open and Distance Learning (DODL), University of Kalyani, Kalyani, Nadia-741235, West Bengal, India}
\email{nurrejwana@gmail.com}


\maketitle

\begin{abstract}
A space $X$ is H-separable (Bella et al., 2009) if for every sequence $(Y_n)$ of dense subspaces of $X$ there exists a sequence $(F_n)$ such that for each $n$ $F_n$ is a finite subset of $Y_n$ and every nonempty open set of $X$ intersects $F_n$ for all but finitely many $n$. In this paper, we introduce and study an ideal variant of H-separability, called $\mathcal{I}$-H-separability.
\end{abstract}

\noindent{\bf\keywordsname{}:} {Selective separable, M-separable, H-separable, $\mathcal{I}$-H-separable.}

\section{Introduction}
In \cite{coc6}, Marion Scheepers introduced and studied the notion of selective separability (also called M-separability; see \cite{bella09}) as a generalization of separability, and the study was later continued in \cite{bella08,bella09,barman11}. It is worth noting that $C_p(X)$ and $2^\kappa$ played crucial roles in the study of selective separability.

A family $\mathcal{I}\subseteq 2^A$ of subsets of a nonempty set $A$ is said to be an ideal on $A$ if $(i)$ $B, C\in\mathcal{I}$ imply $B\cup C\in\mathcal{I}$ and $(ii)$ $B\in\mathcal{I}$ and $C\subseteq B$ imply $C\in\mathcal{I}$, while an admissible ideal $\mathcal{I}$ of $A$ further satisfies $\{a\}\in\mathcal{I}$ for each $a\in A$. Throughout the paper $\mathcal{I}$ stands for a proper admissible ideal on $\mathbb{N}$. The symbol $\fin$ denotes the ideals of all finite subsets of $\mathbb{N}$.

As a variant of selective separability, the notion of H-separability was introduced and studied by Bella et al. \cite{bella09}. Motivated by this, we generalize H-separability by introducing its ideal variant, \(\mathcal{I}\)-H-separability. We establish that if \(\mathcal{I}_1\) and \(\mathcal{I}_2\) are ideals on \(\mathbb{N}\) with \(\mathcal{I}_1 \le_{1\text{-}1} \mathcal{I}_2\), then \(\mathcal{I}_1\)-H-separability implies \(\mathcal{I}_2\)-H-separability. We show that \(\mathcal{I}\)-H-separability lies strictly between H-separability and M-separability. We further show that \(\mathcal{I}\)-H-separability and H-separability are equivalent for a certain class of ideals. The $\mathcal{I}$-H-separability possesses well-behaved categorical properties: it remains invariant under open continuous mappings, closed irreducible mappings, dense subspaces, and open subspaces. However, it does not remain invariant under continuous mappings (see Example~\ref{ex3}). It is observed that the product of two countable $\mathcal{I}$-H-separable spaces remains $\mathcal{I}$-H-separable provided that one of the factors has $\pi$-weight less than $\mathfrak{b}$. For a compact space we show that $\mathcal{I}$-H-separability is equivalent to $\pi$-weight of the space is countable.

The cardinal number $\mathfrak{b}(\mathcal{I})$ and the space $2^\kappa$ are used to obtain certain fascinating observations. We begin by showing that, under a certain assumption, a space with $\pi$-weight less than $\mathfrak{b} (\mathcal{I})$ is $\mathcal{I}$-H-separable. As a consequence of this result, we establish that every countable subspace of $2^\kappa$, where $\kappa < \mathfrak{b} (\mathcal{I})$, is $\mathcal{I}$-H-separable. Furthermore, we construct a countable dense non-$\mathcal{I}$-H-separable subspace of $2^{\mathfrak{b} (\mathcal{I})}$. This allows us to determine that the smallest $\pi$-weight of a countable non-$\mathcal{I}$-H-separable space is precisely $\mathfrak{b} (\mathcal{I})$. It is also demonstrated that $2^{\omega_1}$ contains a countable dense $\mathcal{I}$-H-separable subspace.

Finally, we consider the function space $C_p(X)$ to further our investigation in this direction. To this end, we introduce the notions of $\mathcal{I}$-weakly Fr\'{e}chet in the strict sense (abbreviated as $\iwfs$) and $\mathcal{I}$-weakly Fr\'{e}chet in the strict sense with respect to dense subspaces (abbreviated as $\iwfsd$). On $C_p(X)$, $\mathcal{I}$-H-separability is characterized by the $\iwfs$ and $\iwfsd$ properties, as well as by the condition that the $i$-weight of $X$ is countable and every finite power of $X$ is $\mathcal{I}$-Hurewicz. If $X$ is zero-dimensional, then we observe that $C_p(X)$ is $\mathcal{I}$-H-separable is equivalent to $C_p(X, 2)$, $C_p(X, \mathbb{Z})$, and $C_p(X, \mathbb{Q})$ are $\mathcal{I}$-H-separable. Additionally, we obtain several interesting results concerning $C_p(X)$. Our investigation also leaves some problems for future study.

\section{Preliminaries}
By a space we always mean a topological space. All spaces are assumed to be Tychonoff; otherwise will be mentioned. For undefined notions and terminologies see \cite{Engelking}.
$w(X)$ denotes the weight of a space $X$. A collection $\mathcal{B}$ of open sets of $X$ is called $\pi$-base if every nonempty open set in $X$ contains a nonempty member of $\mathcal{B}$. $\pi w(X) = \min \{|\mathcal{B}| : \mathcal{B} \text{ is a } \pi\text{-base for } X\}$ denotes the $\pi$-weight of $X$. $d(X)$ denotes the minimal cardinality of a dense subspace of $X$ and $\delta(X) = \sup \{d(Y) : Y \text{ is dense in }X\}$. For any space $X$, $d(X) \leq \delta(X) \leq \pi w(X) \leq w(X)$. The $i$-weight of a space $(X, \tau)$ is defined as \[iw((X, \tau)) = \min \{\kappa : \text{there is a Tychonoff topology } \tau^\prime \subseteq \tau \text{ such that } w((X, \tau^\prime)) = \kappa\}.\]

A space $X$ has countable fan tightness \cite{arhan92} if for any $x\in X$ and any
sequence $(Y_n)$ of subsets of $X$ with $x\in \cap_{n\in \mathbb{N}} \overline{Y_n}$ there exists a sequence $(F_n)$ such that for each $n$ $F_n$ is a finite subset of $Y_n$ and $x \in \overline{\cup_{n\in \mathbb{N}} F_n}$. $X$ has countable fan tightness with respect to dense subspaces \cite{bella09} if for any $x\in X$ and any sequence $(Y_n)$ of dense subspaces of $X$ there exists a sequence $(F_n)$ such that for each $n$ $F_n$ is a finite subset of $Y_n$ and $x \in \overline{\cup_{n\in \mathbb{N}} F_n}$. A space $X$ has countable tightness (which is denoted by $t(X) = \omega$) if for each $x\in X$ and each $Y \subseteq X$ with $x\in \overline{Y}$ there exists a countable set $E \subseteq Y$ such that $x\in \overline{E}$. A space $X$ is scattered if every nonempty subspace $Y$ of $X$ has an isolated point. For a space $X$, $\beta X$ denotes the Stone-\v{C}ech compactification of $X$.

Let $X$ be a space and $C(X)$ be the set of all continuous real valued functions. As usual $C_p(X)$ denotes the space $C(X)$ with pointwise convergence topology. Let $f \in C(X)$. Then a basic open set of $f$ in $C_p(X)$ is of the form \[B(f, F, \epsilon) = \{g\in C(X) : |f(x) - g(x)| < \epsilon\; \forall x\in F\},\] where $F$ is a finite subset of $X$ and $\epsilon > 0$. We use $\underbar{1}$ to denote the function which takes the constant value $1$ everywhere on $X$. For a separable metrizable space $X$, $C_p(X)$ is hereditarily separable.

A space $X$ is said to be M-separable \cite{bella08} (see also \cite{bella09}) if for every sequence $(Y_n)$ of dense subspaces of $X$ there exists a sequence $(F_n)$ such that for each $n$ $F_n$ is a finite subset of $Y_n$ and $\cup_{n\in \mathbb{N}} F_n$ is dense in $X$. $X$ is said to be H-separable \cite{bella09} if for every sequence $(Y_n)$ of dense subspaces of $X$ there exists a sequence $(F_n)$ such that for each $n$ $F_n$ is a finite subset of $Y_n$ and every nonempty open set of $X$ intersects $F_n$ for all but finitely many $n$.

A space $X$ is said to have the $\mathcal{I}$-Hurewicz property \cite{PD} (see also \cite{PD4}) if for each sequence $(\mathcal{U}_n)$ of open covers of $X$ there exists a sequence $(\mathcal{V}_n)$ such that for each $n$ $\mathcal{V}_n$ is a finite subset of $\mathcal{U}_n$ and for each $x\in X$, $\{n\in\mathbb{N} : x\notin \cup \mathcal{V}_n\} \in \mathcal{I}$.

\section{Main Results}
\subsection{$\mathcal{I}$-H-separability}
We now introduce the main definition of the paper.
\begin{Def}\rm
\label{def1}
A space $X$ is said to be $\mathcal{I}$-H-separable if for every sequence $(Y_n)$ of dense subspaces of $X$ there exists a sequence $(F_n)$ such that for each $n$ $F_n$ is a finite subset of $Y_n$ and for every nonempty open set $U$ of $X$, $\{n\in \mathbb{N} : U\cap F_n = \emptyset\} \in \mathcal{I}$.
\end{Def}

For a function $\varphi \in \mathbb{N}^\mathbb{N}$ and $\mathcal{A}_1, \mathcal{A}_2 \subseteq 2^\mathbb{N}$, $\mathcal{A}_1 \leq_\varphi \mathcal{A}_2$ implies that $\varphi^{-1}(A) \in \mathcal{A}_2$ for all $A\in \mathcal{A}_1$. $\mathcal{A}_1 \leq_{1\mbox{-}1} \mathcal{A}_2$ implies that there exists a one-one function $\varphi \in \mathbb{N}^\mathbb{N}$ such that $\mathcal{A}_1 \leq_\varphi \mathcal{A}_2$.

\begin{Th}
\label{thm8}
Let $\mathcal{I}_1$ and $\mathcal{I}_2$ be two ideals on $\mathbb{N}$ such that $\mathcal{I}_1 \leq_{1\mbox{-}1} \mathcal{I}_2$. If $X$ is $\mathcal{I}_1$-H-separable, then $X$ is $\mathcal{I}_2$-H-separable.
\end{Th}
\begin{proof}
Since $\mathcal{I}_1 \leq_{1\mbox{-}1} \mathcal{I}_2$, there exists a one-one function $\varphi \in \mathbb{N}^\mathbb{N}$ such that $\varphi^{-1}(A) \in \mathcal{I}_2$ for all $A\in \mathcal{I}_1$. Let $(Y_n)$ be a sequence of dense subspaces of $X$. We define a sequence $(Z_n)$ of dense subspaces of $X$ by \[Z_n = \begin{cases}
          Y_m, & \mbox{if } n = \varphi(m)\\
          Y_1, & \mbox{otherwise}.
        \end{cases}\]
Since $X$ is $\mathcal{I}_1$-H-separable, there exists a sequence $(F_n)$ such that for each $n$ $F_n$ is a finite subset of $Z_n$ and for every nonempty open set $U$ of $X$, $\{n\in \mathbb{N} : U\cap F_n = \emptyset\} \in \mathcal{I}_1$. It follows that for every nonempty open set $U$ of $X$, $\varphi^{-1}(\{n\in \mathbb{N} : U\cap F_n = \emptyset\}) \in \mathcal{I}_2$.

For each $n$ let $F_n^\prime = F_{\varphi(n)}$. Clearly for each $n$ $F_n^\prime$ is a finite subset of $Y_n$. We claim that the sequence $(F_n^\prime)$ witnesses that $X$ is $\mathcal{I}_2$-H-separable. Let $U$ be a nonempty open set of $X$. Then $\varphi^{-1}(\{n\in \mathbb{N} : U\cap F_n= \emptyset\}) \in \mathcal{I}_2$. We now show that $\{n\in \mathbb{N} : U\cap F_n^\prime= \emptyset\} \subseteq \varphi^{-1}(\{n\in \mathbb{N} : U\cap F_n= \emptyset\})$. Let $k\in \{n\in \mathbb{N} : U\cap F_n= \emptyset\}$. Then $U\cap F_k^\prime = \emptyset$, i.e. $U\cap F_{\varphi(k)} = \emptyset$. This gives us $\varphi(k) \in \{n\in \mathbb{N} : U\cap F_n= \emptyset\}$, i.e.  $k\in \varphi^{-1}(\{n\in \mathbb{N} : U\cap F_n= \emptyset\})$. Thus $\{n\in \mathbb{N} : U\cap F_n^\prime= \emptyset\} \subseteq \varphi^{-1}(\{n\in \mathbb{N} : U\cap F_n= \emptyset\})$. It follows that $\{n\in \mathbb{N} : U\cap F_n^\prime= \emptyset\} \in \mathcal{I}_2$ and hence $X$ is $\mathcal{I}_2$-H-separable.
\end{proof}

It is clear that every H-separable space is $\mathcal{I}$-H-separable and every $\mathcal{I}$-H-separable space is M-separable. However, the converses do not hold, as demonstrated in the next two examples (Examples~\labelcref{ex1,ex2}). To proceed, we first recall the following result of \cite{blass99} (see \cite[Theorem 2.1]{blass99}). There are models of ZFC in which $\mathfrak{b} \ll \mathfrak{d}$, and every regular cardinal $\kappa$ between them arises as $\kappa = \mathfrak{b} (\mathcal{I})$ for some maximal ideal $\mathcal{I}$.

\begin{Ex}
\label{ex2}
There exists an $\mathcal{I}$-H-separable space which is not H-separable.
\end{Ex}
\begin{proof}
Let $\mathcal{I}$ be a maximal ideal such that $\mathfrak{b} < \mathfrak{b} (\mathcal{I})$. By \cite[Theorem 31]{bella09}, there exists a countable subspace $Y$ of $2^\mathfrak{b}$ such that $Y$ is not H-separable. Since $Y$ is countable, $\delta(Y) = \omega$. Also $\pi w(Y) = \mathfrak{b} < \mathfrak{b} (\mathcal{I})$. By Theorem~\ref{thm1}, $Y$ is $\mathcal{I}$-H-separable.
\end{proof}

\begin{Ex}
\label{ex1}
There exists a M-separable space which is not $\mathcal{I}$-H-separable.
\end{Ex}
\begin{proof}
Let $\mathcal{I}$ be a maximal ideal such that $\mathfrak{b}(\mathcal{I}) < \mathfrak{d}$. By Theorem~\ref{thm2}, there exists a countable subspace $Y$ of $2^{\mathfrak{b} (\mathcal{I})}$ such that $Y$ is not $\mathcal{I}$-H-separable. Since $Y$ is countable, $\delta(Y) = \omega$. Also $\pi w(Y) = \mathfrak{b} (\mathcal{I}) < \mathfrak{d}$. By \cite[Theorem 11]{bella09} (see also \cite{bella08}), $Y$ is M-separable.
\end{proof}

The above two examples are consistent with ZFC. So it is natural to ask the following questions.

\begin{Prob}
\label{prob1}
In ZFC, does there exist an $\mathcal{I}$-H-separable space which is not H-separable?
\end{Prob}

\begin{Prob}
\label{prob2}
In ZFC, does there exist a M-separable space which is not $\mathcal{I}$-H-separable?
\end{Prob}

We say that a set $C$ is almost contained in a set $D$ if $C\setminus D$ is finite and we use $C\subseteq^* D$ to denote it. Let $\mathcal{C}$ be a family of subsets of $\mathbb{N}$. A subset $D$ of $\mathbb{N}$ is said to be a pseudounion of $\mathcal{C}$ if $\mathbb{N} \setminus D$ is infinite and $C \subseteq^* D$ for every $C \in \mathcal{C}$.

Clearly $\fin$-H-separable $=$ H-separable. We now show that $\mathcal{I}$-H-separability and H-separability are equivalent for certain class of ideals.

\begin{Th}
\label{thm10}
Let $\mathcal{I}$ be an ideal on $\mathbb{N}$ having a pseudounion. Then a space $X$ is H-separable if and only if it is $\mathcal{I}$-H-separable.
\end{Th}
\begin{proof}
For $\mathcal{I} = \fin$, the result follows immediately. Assume that $\mathcal{I} \neq \fin$. We only prove the non-trivial part. Suppose that $X$ is an $\mathcal{I}$-H-separable space. Let $D$ be a pseudounion of $\mathcal{I}$. Let $(Y_n)$ be a sequence of dense subspaces of $X$. Since $X$ is $\mathcal{I}$-H-separable, there exists a sequence $(F_n)$ such that for each $n$ $F_n$ is a finite subset of $Y_n$ and for every nonempty open set $U$ of $X$, \[A_U = \{n \in \mathbb{N} : U \cap F_n = \emptyset\} \in \mathcal{I}.\]

Let $E = \mathbb{N} \setminus D$. Since $D$ is a pseudounion of $\mathcal{I}$, $E$ is infinite and for every $C \in \mathcal{I}$, we have $C \subseteq^* D$, i.e. $C \setminus D$ finite. In particular, for every nonempty open $U$ of $X$, $A_U \cap E$ is finite. Thus, on $E$, the sequence $(F_n)$ already satisfies that for every nonempty open set $U$ of $X$, $\{n \in E : U \cap F_n = \emptyset\}$ is finite.

We now construct a new sequence $(G_n)$ that will work on all of $\mathbb{N}$. Since $D$ and $E$ are both infinite, there exists a bijection $f : D \to E$. Define a new sequence $(Z_n)$ of dense subspaces of $X$ as follows:
\[Z_n =
\begin{cases}
Y_n & \text{if } n \in D, \\
Y_{f^{-1}(n)} & \text{if } n \in E.
\end{cases}\]

Apply the $\mathcal{I}$-H-separability of $X$ to the sequence $(Z_n)$. This yields a sequence $(H_n)$ such that for each $n$ $H_n$ is a finite subset of $Z_n$ and for every nonempty open set $U$ of $X$, \[B_U = \{n \in \mathbb{N} : U \cap H_n = \emptyset\}\in \mathcal{I}.\] Again, since $D$ is a pseudounion of $\mathcal{I}$, $B_U \cap E$ is finite for every nonempty open set $U$ of $X$.

Now define the sequence $(G_n)$ by
\[G_n =
\begin{cases}
F_n & \text{if } n \in E, \\
H_{f(n)} & \text{if } n \in D.
\end{cases}\]
For each $n\in \mathbb{N}$, $G_n$ is a finite subset of $Y_n$ as $(i)$ if $n \in E$, then $G_n = F_n \subseteq Y_n$; $(ii)$ if $n \in D$, then $f(n) \in E$, and $G_n = H_{f(n)} \subseteq Z_{f(n)} = Y_{f^{-1}(f(n))} = Y_n$. We claim that the sequence $(G_n)$ witnesses for $(Y_n)$ that $X$ is H-separable.

Let $U$ be a nonempty open subset of $X$. We show that the set $C_U = \{n \in \mathbb{N} : U \cap G_n = \emptyset\}$ is finite. Clearly $C_U = (C_U \cap E) \cup (C_U \cap D)$. Now \[C_U \cap E = \{n \in E : U \cap G_n = \emptyset\} = \{n \in E : U \cap F_n = \emptyset\} = A_U \cap E,\] which is finite. Also
\begin{alignat*}{1}
C_U \cap D &= \{n \in D : U \cap G_n = \emptyset\} = \{n \in D : U \cap H_{f(n)} = \emptyset \} \\
        &= f^{-1}(\{ m \in E : U \cap H_m = \emptyset \}) = f^{-1}(B_U \cap E),
\end{alignat*}
which is finite, as $f$ is a bijection from $D$ onto $E$ and $B_U \cap E$ is finite. Thus $C_U$ is finite for every nonempty open set $U$ of $X$, and hence $X$ is H-separable.
\end{proof}

Next we present some fundamental observations in the context of $\mathcal{I}$-H-separable spaces. Recall that a continuous mapping $f: X \to Y$ is irreducible if the only closed subset $C$ of $X$ satisfying $f(C) = Y$ is $C = X$.
\begin{Lemma}
\label{lemma3}
The $\mathcal{I}$-H-separability is preserved under open continuous as well as closed irreducible mappings.
\end{Lemma}

\begin{Lemma}
\label{lemma103}
The $\mathcal{I}$-H-separability is preserved under dense as well as open subspaces.
\end{Lemma}

\begin{Lemma}
\label{prop2}
For a space the following assertions are equivalent.
\begin{enumerate}[wide=0pt,label={\upshape(\arabic*)},leftmargin=*]
  \item $X$ is hereditarily $\mathcal{I}$-H-separable.
  \item $X$ is hereditarily separable and every countable subspace of $X$ is $\mathcal{I}$-H-separable.
\end{enumerate}
\end{Lemma}

\begin{Th}
\label{thm12}
If $X$ is a countable $\mathcal{I}$-H-separable space and $Y$ is a countable space with $\pi w(Y) < \mathfrak{b}$, then $X\times Y$ is $\mathcal{I}$-H-separable.
\end{Th}
\begin{proof}
Let $\kappa < \mathfrak{b}$ and $\mathcal{B} = \{B_\alpha : \alpha < \kappa\}$ be a $\pi$-base of $Y$. Pick a sequence $(Y_n)$ of dense subspaces of $X\times Y$. For each $n$ we can choose $Y_n = \{x_m^{(n)} : m\in \mathbb{N}\}$. Let $p_1: X\times Y\to X$ and $p_2: X\times Y\to Y$ be the projection mappings on $X$ and $Y$ respectively. Let $\alpha < \kappa$ be fixed. For each $n$ let $Z_n^{(\alpha)} = p_1(Y_n\cap (X\times B_\alpha))$. Clearly each $Z_n$ is dense in $X$. Since $X$ is $\mathcal{I}$-H-separable, there exists a sequence $(F_n^{(\alpha)})$ such that for each $n$ $F_n^{(\alpha)}$ is a finite subset of $Z_n^{(\alpha)}$ and for every nonempty open set $U$ of $X$, $\{n\in \mathbb{N} : U\cap F_n^{(\alpha)} = \emptyset\} \in \mathcal{I}$. For each $n$ choose a finite subset $G_n^{(\alpha)}$ of $Y_n\cap (X\times B_\alpha)$ such that $F_n^{(\alpha)} = p_1(G_n^{(\alpha)})$. Choose a $f_\alpha \in \mathbb{N}^\mathbb{N}$ such that $G_n^{(\alpha)} \subseteq \{x_m^{(n)} : m \leq f_\alpha(n)\}$. Thus we obtain a subset $\{f_\alpha : \alpha < \kappa\}$ of $\mathbb{N}^\mathbb{N}$. Since $\kappa < \mathfrak{b}$, there exists a $f\in \mathbb{N}^\mathbb{N}$ such that $f_\alpha \leq^* f$ for all $\alpha < \kappa$. For each $n$ let $F_n = \{x_m^{(n)} : m\leq f(n)\}$. Observe that the sequence $(F_n)$ witnesses for $(Y_n)$ that $X\times Y$ is $\mathcal{I}$-H-separable. Let $U$ be a nonempty open set of $X\times Y$. Choose a nonempty open set $V$ of $X$ and a $\alpha < \kappa$ such that $V \times B_\alpha \subseteq U$. Then $\{n\in \mathbb{N} : V\cap F_n^{(\alpha)} = \emptyset\} \in \mathcal{I}$. Choose a $n_0\in \mathbb{N}$ be such that $f_\alpha(n) \leq f(n)$ for all $n\geq n_0$. We claim that $\{n\geq n_0 : (V \times B_\alpha) \cap F_n = \emptyset\} \subseteq \{n\in \mathbb{N} : V\cap F_n^{(\alpha)} = \emptyset\}$. Pick a $k \in \{n\geq n_0 : (V \times B_\alpha) \cap F_n = \emptyset\}$. Then $k\geq n_0$ and $(V \times B_\alpha) \cap F_k = \emptyset$. If possible suppose that $V\cap F_k^{(\alpha)} \neq \emptyset$. Let $x\in V\cap F_k^{(\alpha)}$. Now $x\in F_k^{(\alpha)}$ gives $x = p_1(x, y)$ for some $(x, y) \in G_k^{(\alpha)}$. It follows that $y\in B_\alpha$ and $(x, y) = x_m^{(k)}$ for some $m \leq f_\alpha(k)$. Since $f_\alpha(k) \leq f(k)$, from the construction of $F_k$ we can say that $(x , y) \in F_k$. Also since $(x, y) \in V\times B_\alpha$, $(V \times B_\alpha) \cap F_k \neq \emptyset$. Which is absurd. Thus $V\cap F_k^{(\alpha)} = \emptyset$, i.e. $k\in \{n\in \mathbb{N} : V\cap F_n^{(\alpha)} = \emptyset\}$. It follows that $\{n\geq n_0 : (V \times B_\alpha) \cap F_n = \emptyset\} \subseteq \{n\in \mathbb{N} : V\cap F_n^{(\alpha)} = \emptyset\}$. Consequently $\{n\in \mathbb{N} : (V \times B_\alpha) \cap F_n = \emptyset\} \in \mathcal{I}$, i.e. $\{n\in \mathbb{N} : U \cap F_n = \emptyset\} \in \mathcal{I}$.  Thus $X\times Y$ is $\mathcal{I}$-H-separable.
\end{proof}

\begin{Th}
\label{thm13}
Let $\{X_n : n\in \mathbb{N}\}$ be a family of spaces and $X = \prod_{n\in \mathbb{N}} X_n$. If each $n$ $X_n^\prime = \prod_{k=1}^n X_k$ is $\mathcal{I}$-H-separable, then $X$ is $\mathcal{I}$-H-separable.
\end{Th}
\begin{proof}
Let $(Y_n)$ be a sequence of dense subspaces of $X$. Let $\{N_n : n\in \mathbb{N}\}$ be a partition of $\mathbb{N}$ into disjoint infinite subsets. For each $n$ let $p_n : X\to X_n^\prime$ be the natural projection. Fix $n$. Apply the $\mathcal{I}$-H-separability of $X_n^\prime$ to $(p_n(Y_k) : k\in \mathbb{N})$ to obtain a sequence $(F_k^{(n)} : k\in \mathbb{N})$ such that for each $k$ $F_k^{(n)}$ is a finite subset of $Y_k$ and for every nonempty open set $U$ of $X_n^\prime$, $\{k\in \mathbb{N} : U\cap p_n(F_k^{(n)}) = \emptyset\} \in \mathcal{I}$. For each $n$ choose $F_n = \cup_{k\leq n} F_n^{(k)}$. Thus we obtain a sequence $(F_n)$ such that for each $n$ $F_n$ is a finite subset of $Y_n$. Observe that $(F_n)$ witnesses for $(Y_n)$ that $X$ is $\mathcal{I}$-H-separable. Let $U$ be a nonempty open set of $X$. Choose a basic member $V$ of $X$ such that $V \subseteq U$. Then $V = \prod_{n\in \mathbb{N}} V_n$ with each $V_n$ is open in $X_n$ and $V_n = X_n$ for all $n\geq m$ for some $m\in \mathbb{N}$. Now $\{n\in \mathbb{N} : p_m(V) \cap p_m(F_n^{(m)}) = \emptyset\} \in \mathcal{I}$. We claim that $\{n\geq m : U\cap F_n = \emptyset\} \subseteq \{n\in \mathbb{N} : p_m(V) \cap p_m(F_n^{(m)}) = \emptyset\}$. Pick a $k \in \{n\geq m : U\cap F_n = \emptyset\}$. Then $k\geq m$ and $U\cap F_k = \emptyset$. If possible suppose that $p_m(V) \cap p_m(F_k^{(m)}) \neq \emptyset$. It follows that $\prod_{i=1}^m V_i \cap p_m(F_k^{(m)}) \neq \emptyset$. Now $\prod_{i=1}^m V_i \cap p_m(F_k^{(m)}) \neq \emptyset$ gives a $x = (x_n) \in F_k^{(m)}$ such that $(x_i : 1\leq i \leq m) \in \prod_{i=1}^m V_i$. Since $V_n = X_n$ for all $n\geq m$, $x\in V$. So $V\cap F_k^{(m)} \neq \emptyset$, i.e. $U\cap F_k^{(m)} \neq \emptyset$. Since $F_k = \cup_{i\leq k} F^{(i)}_k$ and $m\leq k$, $U\cap F_k \neq \emptyset$. Which is absurd. Thus $p_m(V) \cap p_m(F_k^{(m)}) = \emptyset$, i.e. $k\in \{n\in \mathbb{N} : p_m(V) \cap p_m(F_n^{(m)}) = \emptyset\}$. This gives $\{n\geq m : U\cap F_n = \emptyset\} \subseteq \{n\in \mathbb{N} : p_m(V) \cap p_m(F_n^{(m)}) = \emptyset\}$ and consequently $\{n\in \mathbb{N} : U\cap F_n = \emptyset\} \in \mathcal{I}$. Thus $X$ is $\mathcal{I}$-H-separable.
\end{proof}

We now shift our focus to compact $\mathcal{I}$-H-separable spaces and begin by recalling the following result.
\begin{Th}[{\cite{juhas89}}]
\label{thm9}
If $X$ is any compact Hausdorff space, then $X$ has a dense subspace $Y$ with $d(Y) = \pi w(X)$.
\end{Th}

\begin{Prop}
\label{prop6}
A compact space $X$ is $\mathcal{I}$-H-separable if and only if $\pi w(X) = \omega$.
\end{Prop}
\begin{proof}
Assume that $X$ is $\mathcal{I}$-H-separable. By Theorem~\ref{thm9}, there exists a dense subspace $Y$ of $X$ such that $d(Y) = \pi w(X)$. Then by Lemma~\ref{lemma103}, $Y$ is $\mathcal{I}$-H-separable and so $Y$ is separable, i.e. $d(Y) = \omega$. Thus $\pi w(X) = \omega$.

Conversely suppose that $\pi w(X) = \omega$. This gives $\delta (X) = \omega$. By Theorem~\ref{thm1}, $X$ is $\mathcal{I}$-H-separable.
\end{proof}

It is important to mention that compactness is essential in the above result. Indeed, since $iw([0,1]) = \omega$ and every finite power of $[0,1]$ is $\mathcal{I}$-Hurewicz, $C_p([0,1])$ is $\mathcal{I}$-H-separable (see Theorem~\ref{thm5}). But $\pi w(C_p([0,1])) = \mathfrak{c}$.

\begin{Prop}
\label{prop7}
Let $X$ be a separable compact space. If $X$ is either scattered or has countable tightness, then every continuous image of $X$ is $\mathcal{I}$-H-separable.
\end{Prop}
\begin{proof}
Let $Y$ be a continuous image of $X$. If $X$ is scattered, then $Y$ is scattered and separable. It follows that $\pi w(Y) = \omega$ and hence by Proposition~\ref{prop6}, $Y$ is $\mathcal{I}$-H-separable.

Now assume that $X$ has countable tightness, i.e. $t(X) = \omega$. This gives $t(Y) = \omega$ (see \cite[1.1.1]{arhan78}) and consequently $\pi_\chi(Y) = \omega$ (see \cite{sapir75}). From the separability of $Y$ we can say that $\pi w(Y) = \omega$. Again by Proposition~\ref{prop6}, $Y$ is $\mathcal{I}$-H-separable.
\end{proof}

\begin{Cor}
\label{cor10}
If a compact space $X$ is either scattered or has countable tightness, then every separable subspace of $X$ is $\mathcal{I}$-H-separable.
\end{Cor}

Thus if $X$ is a countable non-$\mathcal{I}$-H-separable space, then $X$ can not be embedded in a compact space of countable tightness.

In the following example we show that a continuous image of a compact $\mathcal{I}$-H-separable space may not be $\mathcal{I}$-H-separable.
\begin{Ex}
\label{ex3}
There exists a compact $\mathcal{I}$-H-separable space which has a non-$\mathcal{I}$-H-separable continuous image.
\end{Ex}
\begin{proof}
Consider the space $\beta \mathbb{N}$. Since $\pi w(\beta \mathbb{N}) = \omega$ (see \cite{kkjv}), by Proposition~\ref{prop6}, $\beta \mathbb{N}$ is $\mathcal{I}$-H-separable. The Tychonoff cube $\mathbb{I}^\mathfrak{c}$ can be obtained as a continuous image of $\beta \mathbb{N}$, where $\mathbb{I} = [0,1]$. $\mathbb{I}^\mathfrak{c}$ has a countable dense subspace which is not $\mathcal{I}$-H-separable (see \cite[Example 2.14]{bella08} for explanation). By Lemma~\ref{lemma103}, $\mathbb{I}^\mathfrak{c}$ is not $\mathcal{I}$-H-separable.
\end{proof}

Note that every continuous image of a compact space $X$ is $\mathcal{I}$-H-separable does not guarantee that $t(X) = \omega$. Indeed, there exists a separable compact scattered space $X$ with $t(X) \geq \omega_1$ (see \cite[Example 2.8]{bella08}). By Proposition~\ref{prop7}, every continuous image of $X$ is $\mathcal{I}$-H-separable.

\subsection{The cardinal number $\mathfrak{b}(\mathcal{I})$ and the space $2^\kappa$}
\begin{Th}
\label{thm1}
If $\delta(X) = \omega$ and $\pi w(X) < \mathfrak{b} (\mathcal{I})$, then $X$ is $\mathcal{I}$-H-separable.
\end{Th}
\begin{proof}
Let $\mathcal{B} = \{U_\alpha : \alpha < \kappa\}$ be a $\pi$-base for $X$ with $U_\alpha \neq \emptyset$ for all $\alpha < \kappa$. Assume that $\kappa < \mathfrak{b} (\mathcal{I})$. Let $(Y_n)$ be a sequence of dense subspaces of $X$. Since $\delta(X) = \omega$, we get a sequence $(Z_n)$ such that for each $n$ $Z_n = \{x_m^{(n)} : m\in \mathbb{N}\}$ is a dense subspace of $Y_n$. For each $\alpha< \kappa$ we define a $f_\alpha \in \mathbb{N}^\mathbb{N}$ by $f_\alpha(n) = \min \{m\in \mathbb{N} : x_m^{(n)} \in U_\alpha\}$. Since $\kappa < \mathfrak{b} (\mathcal{I})$, there exists a $g\in \mathbb{N}^\mathbb{N}$ such that $\{n\in \mathbb{N} : f_\alpha(n) > g(n)\} \in \mathcal{I}$ for all $\alpha < \kappa$. For each $n$ let $F_n = \{x_m^{(n)} : m\leq g(n)\}$. Observe that the sequence $(F_n)$ witnesses for $(Y_n)$ that $X$ is $\mathcal{I}$-H-separable. Let $U$ be a nonempty open set of $X$. Then there exists a $\beta < \kappa$ such that $U_\beta \subseteq U$. We claim that $\{n\in \mathbb{N} : U_\beta\cap F_n = \emptyset\} \subseteq \{n\in \mathbb{N} : f_\beta(n) > g(n)\}$. Pick a $k\in \{n\in \mathbb{N} : U_\beta\cap F_n = \emptyset\}$. Then $U_\beta \cap F_k = \emptyset$, i.e. $x_m^{(k)} \notin U_\beta$ for all $m\leq g(k)$. Since $x_{f_\beta(k)} \in U_\beta$ and $x_m^{(k)} \notin U_\beta$ for all $m\leq g(k)$, $f_\beta(k) > g(k)$ and consequently $k\in \{n\in \mathbb{N} : f_\beta(n) > g(n)\}$. Thus $\{n\in \mathbb{N} : U_\beta\cap F_n = \emptyset\} \subseteq \{n\in \mathbb{N} : f_\beta(n) > g(n)\}$. It follows that $\{n\in \mathbb{N} : U\cap F_n = \emptyset\} \in \mathcal{I}$ and hence $X$ is $\mathcal{I}$-H-separable.
\end{proof}

Recall from \cite{kkjv} that for $2^\kappa$ (the Cantor cube of weight $\kappa$), $\pi w(2^\kappa) = \kappa = w(2^\kappa)$.
\begin{Cor}
\label{cor1}
If $\kappa < \mathfrak{b} (\mathcal{I})$, then every countable subspace of $2^\kappa$ is $\mathcal{I}$-H-separable.
\end{Cor}
\begin{proof}
Let $X$ be a countable subspace of $2^\kappa$. Then $\delta(X) = \omega$. Since $w(2^\kappa) = \kappa$ and weight is monotone (i.e. for every subspace $Y$ of a space $X$, $w(Y) \leq w(X)$), $w(X) \leq \kappa$. Also since $\pi w (X) \leq w(X)$, $\pi w(X) \leq \kappa$. Hence the result.
\end{proof}

\begin{Th}
\label{thm2}
The space $2^{\mathfrak{b}(\mathcal{I})}$ contains a countable dense subspace which is not $\mathcal{I}$-H-separable.
\end{Th}
\begin{proof}
Pick a countable dense subspace $X = \{f_n : n\in \mathbb{N}\}$ of $2^{\mathfrak{b}(\mathcal{I})}$ and an $\mathcal{I}$-unbounded subset $\{g_\alpha : \alpha < \mathfrak{b}(\mathcal{I})\}$ of $\mathbb{N}^\mathbb{N}$. For each $n, m \in \mathbb{N}$ define a $h_m^{(n)} \in 2^{\mathfrak{b}(\mathcal{I})}$ by
\[h_m^{(n)} (\alpha) = \begin{cases}
                1, & \mbox{if } m < g_\alpha(n) \\
                f_m (\alpha), & \mbox{otherwise}.
              \end{cases}\]
We claim that for each $n$ $Y_n = \{h_m^{(n)} : m\in \mathbb{N}\}$ is dense in $2^{\mathfrak{b}(\mathcal{I})}$. Fix $n$. Choose a basic open set of $U = \{f\in 2^{\mathfrak{b} (\mathcal{I})} : f \upharpoonright_\kappa = \sigma\}$ of $2^{\mathfrak{b}(\mathcal{I})}$, where $\kappa$ is a finite subset of $\mathfrak{b} (\mathcal{I})$ and $\sigma : \kappa \to 2$ is a function. Define $g = \max \{g_\alpha : \alpha\in \kappa\}$. Since $X$ is dense in $2^{\mathfrak{b}(\mathcal{I})}$, $U\cap X$ is infinite, i.e. there exists a $m> g(n)$ such that $f_m\upharpoonright_\kappa = \sigma$. Pick $\alpha \in \kappa$. Then from the definition of $g$ we get $g_\alpha(n) \leq g(n)$ and so $m > g_\alpha(n)$. This gives us $h_m^{(n)}(\alpha) = f_m(\alpha)$. It follows that $h_m^{(n)} \upharpoonright_\kappa = \sigma$ and hence $U\cap Y_n \neq \emptyset$. Thus $Y_n$ is dense in $2^{\mathfrak{b}(\mathcal{I})}$.

Choose $Y = \cup_{n\in \mathbb{N}} Y_n$. Then $Y$ is dense in $2^{\mathfrak{b}(\mathcal{I})}$. We now claim that the sequence $(Y_n)$ for $Y$ that $Y$ is not $\mathcal{I}$-H-separable. Let $(F_n)$ be a sequence such that for each $n$ $F_n$ is a finite subset of $Y_n$. Let us consider a $f\in \mathbb{N}^\mathbb{N}$ such that $F_n \subseteq \{h_m^{(n)} : m < f(n)\}$. Since $\{g_\alpha : \alpha < \mathfrak{b}(\mathcal{I})\}$ is unbounded, there exists a $\beta < \mathfrak{b}(\mathcal{I})$ such that $\{n\in \mathbb{N} : g_\beta(n) > f(n)\} \notin \mathcal{I}$, i.e. $A = \{n\in \mathbb{N} : g_\beta(n) > f(n)\}$ is infinite. Let $n\in A$. Choose a $h_m^{(n)} \in F_n$. Then $m < f(n) < g_\beta(n)$ and so $h_m^{(n)} (\beta) = 1$. Thus for any $n\in A$ and any $h \in F_n$ we can say that $h(\beta) = 1$. For the basic open set $V = \{h\in 2^{\mathfrak{b}(\mathcal{I})} : h(\beta) = 0\}$ of $2^{\mathfrak{b}(\mathcal{I})}$, we have $V \cap F_n = \emptyset$ for all $n \in A$, i.e. $A \subseteq \{n\in \mathbb{N} : V \cap F_n = \emptyset\}$ and hence  $\{n\in \mathbb{N} : V \cap F_n = \emptyset\} \notin \mathcal{I}$. Since the sequence $(F_n)$ is arbitrarily chosen, $Y$ is not $\mathcal{I}$-H-separable.
\end{proof}

\begin{Cor}
\label{cor2}
The smallest $\pi$-weight of a countable non-$\mathcal{I}$-H-separable space is $\mathfrak{b} (\mathcal{I})$.
\end{Cor}
\begin{proof}
Let $X$ be a countable space. Then $\delta(X) = \omega$. If $\pi w(X) < \mathfrak{b} (\mathcal{I})$, then by Theorem~\ref{thm1}, $X$ is $\mathcal{I}$-H-separable.

Now by Theorem~\ref{thm2}, $2^{\mathfrak{b} (\mathcal{I})}$ contains a countable subspace $Y$  such that $\pi w(Y) = \mathfrak{b} (\mathcal{I})$ and $Y$ is not $\mathcal{I}$-H-separable. Thus the smallest $\pi$-weight of a countable non-$\mathcal{I}$-H-separable space is $\mathfrak{b} (\mathcal{I})$.
\end{proof}

\begin{Th}[{\cite[Theorem 46]{bella09}}]
\label{thm7}
The space $2^\mathfrak{c}$ contains a countable dense H-separable subspace.
\end{Th}

\begin{Cor}
\label{cor8}
The space $2^\mathfrak{c}$ contains a countable dense $\mathcal{I}$-H-separable subspace.
\end{Cor}

A space $X$ is said to have the Hurewicz property \cite{coc1} (see also \cite{coc2}) if for each sequence $(\mathcal{U}_n)$ of open covers of $X$ there exists a sequence $(\mathcal{V}_n)$ such that for each $n$ $\mathcal{V}_n$ is a finite subset of $\mathcal{U}_n$ and each $x\in X$ belongs to $\cup \mathcal{V}_n$ for all but finitely many $n$.

\begin{Th}
\label{thm11}
The space $2^{\omega_1}$ contains a dense H-separable subspace.
\end{Th}
\begin{proof}
Assume that $\omega_1 < \mathfrak{b}$. By \cite[Corollary 30]{bella09}, every countable subspace of $2^{\omega_1}$ is H-separable. This gives a countable dense H-separable subspace of $2^{\omega_1}$.

Next assume that $\omega_1 = \mathfrak{b}$. By \cite[Theorem 10]{bartos06}, there exists a zero-dimensional metrizable space $X$ with $|X| = \mathfrak{b}$ such that every finite power of $X$ is Hurewicz. Clearly $iw(X) = \omega$. By \cite[Proposition 44]{bella09}, $C_p(X, 2)$ is a H-separable subspace of $2^{\omega_1}$. Also $C_p(X, 2)$ is dense in $2^{\omega_1}$.
\end{proof}

\begin{Cor}
\label{cor9}
The space $2^{\omega_1}$ contains a countable dense $\mathcal{I}$-H-separable subspace.
\end{Cor}

\subsection{$\iwfs$, $\iwfsd$ and $C_p(X)$}
We now introduce the following notions. A space $X$ is said to be $\mathcal{I}$-weakly Fr\'{e}chet in the strict sense (in short, $\iwfs$) if for any $x\in X$ and any sequence $(Y_n)$ of subsets of $X$ with $x\in \cap_{n\in \mathbb{N}} \overline{Y_n}$ there exists a sequence $(F_n)$ such that for each $n$ $F_n$ is a finite subset of $Y_n$ and for every neighbourhood $U$ of $x$, $\{n\in \mathbb{N} : U\cap F_n = \emptyset\} \in \mathcal{I}$. $X$ is said to be $\mathcal{I}$-weakly Fr\'{e}chet in the strict sense with respect to dense subspaces (in short, $\iwfsd$) if for any $x\in X$ and any sequence $(Y_n)$ of dense subspaces of $X$ there exists a sequence $(F_n)$ such that for each $n$ $F_n$ is a finite subset of $Y_n$ and for every neighbourhood $U$ of $x$, $\{n\in \mathbb{N} : U\cap F_n = \emptyset\} \in \mathcal{I}$. If $X$ is $\iwfs$, then $X$ is $\iwfsd$. Also if $X$ is $\iwfs$ (respectively, $\iwfsd$), then $X$ has countable fan tightness (respectively, countable fan tightness with respect to dense subspaces).

It is obvious that if $X$ is $\mathcal{I}$-H-separable, then $X$ is $\iwfsd$. But an $\mathcal{I}$-H-separable space may not be $\iwfs$. Indeed, let $X$ be the countable Fr\'{e}chet-Urysohn fan space (see \cite[2.3.1]{arhan78}). Since $\pi w(X) = \omega$, by Theorem~\ref{thm1}, $X$ is $\mathcal{I}$-H-separable. On the other hand, $X$ does not have countable fan tightness and hence $X$ is not $\iwfs$. However, we show in Theorem~\ref{thm5} that if $C_p(X)$ is $\mathcal{I}$-H-separable, then $C_p(X)$ is $\iwfs$.

\begin{Prop}
\label{prop1}
A separable space $X$ is $\mathcal{I}$-H-separable if and only if it is $\iwfsd$.
\end{Prop}
\begin{proof}
Suppose that $X$ is $\iwfsd$. Let $(Y_n)$ be a sequence of dense subspaces of $X$. Since $X$ is separable, we have a countable dense subspace $Y = \{x_n : n\in \mathbb{N}\}$ of $X$. Fix $n\in \mathbb{N}$. Since $X$ is $\iwfsd$, we get a sequence $(F_m^{(n)} : m\in \mathbb{N})$ such that for each $m\in \mathbb{N}$, $F_m^{(n)}$ is a finite subset of $Y_m$ and for each neighbourhood $U$ of $x_n$, $\{m\in \mathbb{N} : U\cap F_m^{(n)} = \emptyset\} \in \mathcal{I}$. For each $n$ let $F_n = \cup_{i\leq n} F_n^{(i)}$. We claim that the sequence $(F_n)$ witnesses for $(Y_n)$ that $X$ is $\mathcal{I}$-H-separable. Let $U$ be a nonempty open set of $X$. Now we can choose a $k\in \mathbb{N}$ such that $x_k \in U \cap Y$. This gives a $\{m\in \mathbb{N} : U\cap F_m^{(k)} = \emptyset\} \in \mathcal{I}$. Since for each $m \geq k$, $F_m^{(k)} \subseteq F_m$, we get $\{m\in \mathbb{N} : U\cap F_m = \emptyset\} \in \mathcal{I}$. Thus $X$ is $\mathcal{I}$-H-separable.

The other direction is trivial.
\end{proof}

\begin{Lemma}
\label{lemma1}
\hfill
\begin{enumerate}[wide=0pt,label={\upshape(\arabic*)},leftmargin=*,ref={\theLemma(\arabic*)}]
  \item\label{lemma101} The $\iwfs$ property is hereditary.
  \item\label{lemma102} The $\iwfsd$ property is preserved under dense subspaces.
\end{enumerate}
\end{Lemma}

Let $X$ be a space and $k\in \mathbb{N}$. Let $\mathcal{U}$ be an open cover of $X^k$. Recall that a family $\mathcal{B}$ of subsets of $X$ is called $\mathcal{U}$-small if for every $V_1, V_2, \dotsc, V_k \in \mathcal{V}$ there exists a $U \in \mathcal{U}$ such that $V_1\times V_2\times \cdots \times V_k \subseteq U$.

We closely follow the technique of the proof of \cite[II.2.2. Theorem]{arhan92} to prove the next result.
\begin{Th}
\label{thm3}
For a space $X$ the following assertions are equivalent.
\begin{enumerate}[wide=0pt,label={\upshape(\arabic*)},leftmargin=*]
  \item $C_p(X)$ is $\iwfs$.
  \item $C_p(X)$ is $\iwfsd$.
  \item Every finite power of $X$ is $\mathcal{I}$-Hurewicz.
\end{enumerate}
\end{Th}
\begin{proof}
$(2) \Rightarrow (3)$. Let $k\in \mathbb{N}$ be fixed. Let $(\mathcal{U}_n)$ be a sequence of open covers of $X^k$. For each $n$ let $\mathfrak{F}_n$ be the collection of all finite $\mathcal{U}_n$-small families of open sets of $X$. For each $\mathcal{U} \in \mathfrak{F}_n$ define $Z_n(\mathcal{U}) = \{f\in C_p(X) : f(X\setminus \cup \mathcal{U}) = \{0\}\}$. For each $n$ choose $Y_n = \cup \{Z_n(\mathcal{U}) : \mathcal{U}\in \mathfrak{F}_n\}$. Observe that for each $n$ $Y_n$ is dense in $C_p(X)$. Since $C_p(X)$ is $\iwfsd$, there exists a sequence $(F_n)$ such that for each $n$ $F_n$ is a finite subset of $Y_n$ and for every neighbourhood $U$ of $\underbar{1}$, $\{n\in \mathbb{N} : U\cap F_n = \emptyset\} \in \mathcal{I}$. For each $n$ and each $f\in F_n$ choose a $\mathcal{U}_f^{(n)} \in \mathfrak{F}_n$ such that $f\in Z_n(\mathcal{U}_f^{(n)})$. For every $V_1, V_2, \dotsc, V_k \in \mathcal{U}_f^{(n)}$ choose a $U_f^{(n)}(V_1, V_2, \dotsc, V_k) \in \mathcal{U}_n$ such that $V_1 \times V_2\times \dotsc V_k \subseteq U_f^{(n)}(V_1, V_2, \dotsc, V_k)$. For each $n$ $\mathcal{V}_n = \{U_f^{(n)}(V_1, V_2, \dotsc, V_k) : V_1, V_2, \dotsc, V_k \in \mathcal{U}_f^{(n)}, f\in F_n\}$ is a finite subset of $\mathcal{U}_n$. Observe that the sequence $(\mathcal{V}_n)$ witnesses for $(\mathcal{U}_n)$ that $X^k$ is $\mathcal{I}$-Hurewicz. Let $x = (x_1, x_2, \dotsc, x_k) \in X^k$. Pick the open set $U = \{f\in C_p(X) : f(x_i) > 0, 1\leq i \leq k\}$ of $C_p(X)$. Clearly $\underbar{1} \in U$. Then $\{n\in \mathbb{N} : U\cap F_n = \emptyset\} \in \mathcal{I}$. We claim that $\{n\in \mathbb{N} : x\notin \cup \mathcal{V}_n\} \subseteq \{n\in \mathbb{N} : U\cap F_n = \emptyset\}$. Pick a $m \in \{n\in \mathbb{N} : x\notin \cup \mathcal{V}_n\}$. Then $x \notin \cup \mathcal{V}_m$. If possible suppose that $U\cap F_m \neq \emptyset$. Pick a $f \in C_p(X)$ such that $f\in U\cap F_m$. Since $f\in Z_m(\mathcal{U}_f^{(m)})$, $f(X\setminus \cup \mathcal{U}_f^{(m)}) = \{0\}$. It follows that $f(x_i) > 0$ for all $1\leq i \leq k$ and $f(y) = 0$ for all $y\in X\setminus \cup \mathcal{U}_f^{(m)}$. For each $1\leq i \leq k$ choose a $V_i \in \mathcal{U}_f^{(m)}$ such that $x_i \in V_i$. So $x \in V_1\times V_2 \times \dotsc V_k$. Then $x\in U_f^{(m)}(V_1, V_2, \dotsc, V_k)$ as $V_1 \times V_2\times \dotsc V_k \subseteq U_f^{(m)}(V_1, V_2, \dotsc, V_k)$ and hence $x\in \cup \mathcal{V}_m$. Which is absurd. Consequently $U\cap F_m \neq \emptyset$, i.e. $m\in \{n\in \mathbb{N} : U\cap F_n = \emptyset\}$ and so $\{n\in \mathbb{N} : x\notin \cup \mathcal{V}_n\} \subseteq \{n\in \mathbb{N} : U\cap F_n = \emptyset\}$, i.e. $\{n\in \mathbb{N} : x\notin \cup \mathcal{V}_n\} \in \mathcal{I}$. Thus $X^k$ is $\mathcal{I}$-Hurewicz.

$(3) \Rightarrow (1)$. Let $f\in C_p(X)$ and $(Y_n)$ be a sequence of sets of $C_p(X)$ such that $f\in \cap_{n\in \mathbb{N}} \overline{Y_n}$. Let $k \in \mathbb{N}$ be fixed. Since $f\in \cap_{n\in \mathbb{N}} \overline{Y_n}$, for each $n\in \mathbb{N}$ and each $x = (x_1, x_2, \dotsc, x_k) \in X^k$ there exists a $g_x^{(n,k)} \in Y_n$ such that $|g_x^{(n,k)}(x_i) - f(x_i)| < \frac{1}{k}$ for all $1\leq i \leq k$. Since $g_x^{(n,k)}$ and $f$ are continuous, for each $1 \leq i \leq k$ we can choose an open set $U_{i,x}^{(n,k)}$ of $X$ containing $x_i$ such that $|g_x^{(n,k)}(y) - f(y)| < \frac{1}{k}$ for all $y \in U_{i,x}^{(n,k)}$. Then $U_x^{(n,k)} = U_{1,x}^{(n,k)}\times U_{2,x}^{(n,k)}\times \dotsc \times U_{k,x}^{(n,k)}$ is an open set in $X^k$ containing $x$ and $|g_x^{(n,k)}(y_i) - f(y_i)| < \frac{1}{k}$ for all $y = (y_1, y_1, \dotsc, y_k) \in U_x^{(n,k)}$. For each $n$ $\mathcal{U}_n^{(k)} = \{U_x^{(n,k)} : x\in X^k\}$ is an open cover of $X^k$. Apply the $\mathcal{I}$-Hurewicz property of $X^k$ to $(\mathcal{U}_n^{(k)})$ to obtain a sequence $(\mathcal{V}_n^{(k)})$ such that for each $n$ $\mathcal{V}_n^{(k)}$ is a finite subset of $\mathcal{U}_n^{(k)}$ and for each $x\in X^k$, $\{n\in \mathbb{N} : x\notin \cup \mathcal{V}_n^{(k)}\} \in \mathcal{I}$. We can obtain a sequence $(A_n^{(k)})$ of finite subsets of $X^k$ such that for each $n$ $\mathcal{V}_n^{(k)} = \{U_x^{(n,k)} : x\in A_n^{(k)}\}$. For each $n$ $F_n^{(k)} = \{g_x^{(n,k)} : x\in A_n^{(k)}\}$ is a finite subset of $Y_n$. We now define a sequence $(F_n)$ such that for each $n$ $F_n$ is a finite subset of $Y_n$ as $F_n = \cup_{k\leq n} F_n^{(k)}$. Observe that the sequence $(F_n)$ witnesses for $f$ and $(Y_n)$ that $C_p(X)$ is $\iwfs$. Let $U$ be an open set in $C_p(X)$ containing $f$. Then we get a finite subset $F = \{z_1, z_2, \dotsc, z_m\}$ of $X$ and an $\epsilon > 0$ such that $B(f,F, \epsilon) \subseteq U$. We can assume that $\frac{1}{m} < \epsilon$. Since $z = (z_1, z_2, \dotsc, z_m) \in X^m$, $\{n\in \mathbb{N} : z\notin \cup \mathcal{V}_n^{(m)}\} \in \mathcal{I}$. We claim that $\{n\in \mathbb{N} : B(f,F, \epsilon) \cap F_n = \emptyset\} \in \mathcal{I}$. It is enough to show that $\{n\geq m : B(f,F, \epsilon) \cap F_n = \emptyset\} \subseteq \{n\in \mathbb{N} : z\notin \cup \mathcal{V}_n^{(m)}\}$. Pick a $j \in \{n\geq m : B(f,F, \epsilon) \cap F_n = \emptyset\}$. Then $j \geq m$ and $B(f,F, \epsilon) \cap F_j = \emptyset$. It follows that $B(f,F, \epsilon) \cap F_j^{(i)} = \emptyset$ for all $i \leq j$, i.e. $B(f,F, \epsilon) \cap F_j^{(m)} = \emptyset$. It follows that for each $x \in A_j^{(m)}$, $g_x^{(j,m)} \notin B(f,F, \epsilon)$ and $|g_x^{(j,m)}(z_p) - f(z_p)| \geq \epsilon > \frac{1}{m}$ for some $1 \leq p \leq m$. Consequently $z \notin U_x^{(j,m)}$ for all $x \in A_j^{(m)}$, i.e. $z\notin \cup \mathcal{V}_j^{(m)}$ and $j \in \{n\in \mathbb{N} : z\notin \cup \mathcal{V}_n^{(m)}\}$. Thus $\{n\geq m : B(f,F, \epsilon) \cap F_n = \emptyset\} \subseteq \{n\in \mathbb{N} : z\notin \cup \mathcal{V}_n^{(m)}\}$ and $\{n\in \mathbb{N} : B(f,F, \epsilon) \cap F_n = \emptyset\} \in \mathcal{I}$, i.e. $\{n\in \mathbb{N} : U \cap F_n = \emptyset\} \in \mathcal{I}$. Hence $C_p(X)$ is $\iwfs$.
\end{proof}

\begin{Cor}
\label{cor3}
If every finite power of $X$ is $\mathcal{I}$-Hurewicz, then every separable subspace of $C_p(X)$ is $\mathcal{I}$-H-separable.
\end{Cor}
\begin{proof}
Let $Y$ be a separable subspace of $C_p(X)$. Suppose that every finite power of $X$ is $\mathcal{I}$-Hurewicz. Then by Theorem~\ref{thm3}, $C_p(X)$ is $\iwfs$. Also by Lemma~\ref{lemma101}, $Y$ is $\iwfs$. Thus $Y$ is $\mathcal{I}$-H-separable (see Proposition~\ref{prop1}).
\end{proof}

\begin{Lemma}[{Folklore}]
\label{lemma2}
Every Lindel\"{o}f space of cardinality less than $\mathfrak{b}(\mathcal{I})$ is $\mathcal{I}$-Hurewicz.
\end{Lemma}
\begin{proof}
Let $X$ be a Lindel\"{o}f space with $|X| < \mathfrak{b}(\mathcal{I})$. Consider a sequence $(\mathcal{U}_n)$ of open covers of $X$. Since $X$ is Lindel\"{o}f, for each $n$ we can assume that $\mathcal{U}_n = \{U_m^{(n)} : m\in \mathbb{N}\}$. For each $x\in X$ choose a $f_x\in\mathbb{N}^\mathbb{N}$ such that $f_x(n) = \min \{m\in \mathbb{N} : x\in U_m^{(n)}\}$ for all $n$. Since $|\{f_x : x\in X\}| < \mathfrak{b} (\mathcal{I})$, there exists a $g\in \mathbb{N}^\mathbb{N}$ such that $\{n\in \mathbb{N} : f_x(n) > g(n)\} \in \mathcal{I}$ for all $x\in X$. For each $n$ let $\mathcal{V}_n = \{U_m^{(n)} : m\leq g(n)\}$. We claim that the sequence $(\mathcal{V}_n)$ witnesses for $(\mathcal{U}_n)$ that $X$ is $\mathcal{I}$-Hurewicz. Let $x\in X$. Pick a $k\in \{n\in \mathbb{N} : x\notin \cup \mathcal{V}_n\}$. So $x\notin U_m^{(k)}$ for all $m\leq g(k)$. It follows that $f_x(k) > g(k)$ because $x\in U_{f_x(k)}^{(k)}$. Thus $k\in \{n\in \mathbb{N} : f_x(n) > g(n)\}$ and $\{n\in \mathbb{N} : x\notin \cup \mathcal{V}_n\} \subseteq \{n\in \mathbb{N} : f_x(n) > g(n)\}$. Consequently $\{n\in \mathbb{N} : x\notin \cup \mathcal{V}_n\} \in \mathcal{I}$ and $X$ is $\mathcal{I}$-Hurewicz.
\end{proof}

\begin{Th}[{\cite{noble74} (see also \cite[I.1.5 Theorem]{arhan92})}]
\label{thm4}
For a space $X$, $iw(X) = d(C_p(X))$.
\end{Th}

\begin{Th}
\label{thm5}
For a space $X$ the following assertions are equivalent.
\begin{enumerate}[wide=0pt,label={\upshape(\arabic*)},leftmargin=*]
  \item $C_p(X)$ is separable and $\iwfs$.
  \item $C_p(X)$ is separable and $\iwfsd$.
  \item $C_p(X)$ is $\mathcal{I}$-H-separable.
  \item $iw(X) = \omega$ and every finite power of $X$ is $\mathcal{I}$-Hurewicz.
\end{enumerate}
\end{Th}
\begin{proof}
$(1) \Rightarrow (2)$ is trivial as $\iwfs$ implies $\iwfsd$.

$(2) \Rightarrow (3)$. By Theorem~\ref{thm3}, every finite power of $X$ is $\mathcal{I}$-Hurewicz. Also by Corollary~\ref{cor3}, $C_p(X)$ is $\mathcal{I}$-H-separable.

$(3) \Rightarrow (4)$. Clearly $C_p(X)$ is $\iwfsd$. By Theorem~\ref{thm3}, every finite power of $X$ is $\mathcal{I}$-Hurewicz. Since $C_p(X)$ is separable, $d(C_p(X)) = \omega$. By Theorem~\ref{thm4}, $iw(X) = \omega$.

$(4) \Rightarrow (1)$. By Theorem~\ref{thm3}, $C_p(X)$ $\iwfs$. Since $iw(X) = \omega$, by Theorem~\ref{thm4}, $d(C_p(X)) = \omega$. It follows that $C_p(X)$ is separable.
\end{proof}

\begin{Cor}
\label{cor4}
If $C_p(X)$ is $\mathcal{I}$-H-separable, then for each $n \in \mathbb{N}$, $C_p(X^n)$ is $\mathcal{I}$-H-separable, and $(C_p(X))^\omega$ is $\mathcal{I}$-H-separable.
\end{Cor}
\begin{proof}
Let $n \in \mathbb{N}$ and $Y = X^n$. Since $C_p(X)$ is $\mathcal{I}$-H-separable, $iw(X) = \omega$ and every finite power of $X$ is $\mathcal{I}$-Hurewicz. It follows that every finite power of $Y$ is also $\mathcal{I}$-Hurewicz and $iw(Y) = \omega$. Thus $C_p(Y)$ is $\mathcal{I}$-H-separable
\end{proof}

\begin{Cor}
\label{cor5}
If $X$ is a second countable space such that $C_p(X)$ is $\mathcal{I}$-H-separable, then $C_p(X)$ is hereditarily $\mathcal{I}$-H-separable.
\end{Cor}
\begin{proof}
Let $Y$ be a subspace of $C_p(X)$. Since $X$ is a second countable space, it is separable metrizable. It follows that $C_p(X)$ is hereditarily separable and hence $Y$ is separable. Since $C_p(X)$ is $\mathcal{I}$-H-separable, $C_p(X)$ is $\iwfs$. By Lemma~\ref{lemma101}, $Y$ is $\iwfs$ and so $\iwfsd$. Also by Proposition~\ref{prop1}, $Y$ is $\mathcal{I}$-H-separable. Thus $C_p(X)$ is hereditarily $\mathcal{I}$-H-separable.
\end{proof}

\begin{Cor}
\label{cor6}
If $X$ is a second countable space with $|X|< \mathfrak{b}(\mathcal{I})$, then $C_p(X)$ is hereditarily $\mathcal{I}$-H-separable.
\end{Cor}
\begin{proof}
Let $Y$ be a subspace of $C_p(X)$. Since $X$ is a second countable space, it is separable metrizable. It follows that $C_p(X)$ is hereditarily separable and hence $Y$ is separable. Also every finite power of $X$ is second countable and so every finite power of $X$ is Lindel\"{o}f. Since $|X|< \mathfrak{b}(\mathcal{I})$, by Lemma~\ref{lemma2}, every finite power of $X$ is $\mathcal{I}$-Hurewicz. It follows that $C_p(X)$ is $\iwfs$ as $iw(X) = d(C_p(X)) = \omega$. Then $Y$ is $\iwfs$ and so $\iwfsd$. Also by Proposition~\ref{prop1}, $Y$ is $\mathcal{I}$-H-separable. Thus $C_p(X)$ is hereditarily $\mathcal{I}$-H-separable.
\end{proof}

\begin{Prop}
\label{prop8}
The space $C_p(C_p(X))$ is $\mathcal{I}$-H-separable if and only if $X$ is finite.
\end{Prop}
\begin{proof}
Suppose that $C_p(C_p(X))$ is $\mathcal{I}$-H-separable. Then by Theorem~\ref{thm5}, every finite power of $C_p(X)$ is $\mathcal{I}$-Hurewicz. Also by \cite[Theorem II.2.10]{arhan92}, $X$ is finite.

Conversely suppose that $X$ is finite. Then $C_p(X)$ is a $\sigma$-compact metrizable space and so $C_p(X)$ is a separable metrizable space. It follows that $C_p(C_p(X))$ is hereditarily separable and so $d(C_p(C_p(X))) = \omega$. By Theorem~\ref{thm4}, $iw(C_p(X)) = \omega$. Since $C_p(X)$ is $\sigma$-compact, every finite power of $C_p(X)$ is $\mathcal{I}$-Hurewicz. Thus by Theorem~\ref{thm5}, $C_p(C_p(X))$ is $\mathcal{I}$-H-separable.
\end{proof}

\begin{Rem}\rm
\label{rem1}
\hfill
\begin{enumerate}[wide=0pt,label={\upshape(\arabic*)},ref={\theRem(\arabic*)}]
\item\label{rem101} Recall that the Baire space $\mathbb{N}^\mathbb{N}$ can be condensed onto a compact space $X$. It follows that $C_p(X)$ can be densely embedded in $C_p(\mathbb{N}^\mathbb{N})$. Now $C_p(X)$ is separable, i.e. $d(C_p(X)) = \omega$ and so $iw(X) = \omega$ (see Theorem~\ref{thm4}). Since $X$ is compact, every finite power of $X$ is $\mathcal{I}$-Hurewicz and hence by Theorem~\ref{thm5}, $C_p(X)$ is $\mathcal{I}$-H-separable. Let $Y$ be a countable dense subspace of $C_p(X)$. Then $Y$ is $\mathcal{I}$-H-separable. Since $C_p(X)$ is densely embeddable in $C_p(\mathbb{N}^\mathbb{N})$, $Y$ is a countable dense $\mathcal{I}$-H-separable subspace of $C_p(\mathbb{N}^\mathbb{N})$. Also since $C_p(\mathbb{N}^\mathbb{N})$ can be densely embedded in the Tychonoff cube $\mathbb{I}^\mathfrak{c}$, $Y$ is a countable dense $\mathcal{I}$-H-separable subspace of $\mathbb{I}^\mathfrak{c}$. Thus $\mathbb{I}^\mathfrak{c}$ has a countable dense $\mathcal{I}$-H-separable subspace.

\item\label{rem102} Moreover, we have already mentioned in Example~\ref{ex3} that $\mathbb{I}^\mathfrak{c}$ has a countable dense non-$\mathcal{I}$-H-separable subspace.
\end{enumerate}
\end{Rem}

\begin{Th}
\label{thm6}
For a zero-dimensional space $X$ the following assertions are equivalent.
\begin{enumerate}[wide=0pt,label={\upshape(\arabic*)},leftmargin=*]
  \item $C_p(X, \mathbb{Q})$ is $\mathcal{I}$-H-separable.
  \item $C_p(X, \mathbb{Z})$ is $\mathcal{I}$-H-separable.
  \item $C_p(X, 2)$ is $\mathcal{I}$-H-separable.
  \item $iw(X) = \omega$ and every finite power of $X$ is $\mathcal{I}$-Hurewicz.
\end{enumerate}
\end{Th}
\begin{proof}
$(1) \Rightarrow (2)$. For each $n \in \mathbb{Z}$ choose an irrational point $x_n \in (n, n+1)$. Let $\psi \in \mathbb{Z}^\mathbb{Q}$ be given by $\psi(x) = n$ if $x\in (x_{n-1}, x_n)$. Now define a mapping $\Phi : C_p(X, \mathbb{Q}) \to C_p(X, \mathbb{Z})$ by $\Phi(f) = \psi \circ f$ for all $f \in C_p(X, \mathbb{Q})$. Observe that $\Phi$ is an open continuous mapping from $C_p(X, \mathbb{Q})$ onto $C_p(X, \mathbb{Z})$. Then by Lemma~\ref{lemma3}, $C_p(X, \mathbb{Z})$ is $\mathcal{I}$-H-separable.

$(2) \Rightarrow (3)$. Let $\phi \in 2^\mathbb{Z}$ be given by
\[\phi(n) = \begin{cases}
             0, & \mbox{if } n< 0 \\
             1, & \mbox{otherwise}.
            \end{cases}\]
Define a mapping $\Psi : C_p(X, \mathbb{Z}) \to C_p(X, 2)$ by $\Psi(f) = \phi \circ f$ for all $f \in C_p(X, \mathbb{Z})$. Clearly $\Psi$ is an open continuous mapping from $C_p(X, \mathbb{Z})$ onto $C_p(X, 2)$. By Lemma~\ref{lemma3}, $C_p(X, 2)$ is $\mathcal{I}$-H-separable.

$(3) \Rightarrow (4)$. Let $D$ be a countable dense subspace of $C_p(X, 2)$. Then $\{f^{-1}(0), f^{-1}(1) : f\in D\}$ forms a base for a zero-dimensional second countable topology on $X$ which is contained in the original topology of $X$. It follows that $iw(X) = \omega$.

We now show that every finite power of $X$ is $\mathcal{I}$-Hurewicz. Let $k\in \mathbb{N}$ be fixed. Let $(\mathcal{U}_n)$ be a sequence of open covers of $X^k$. For each $n$ let $\mathfrak{F}_n$ be the collection of all finite $\mathcal{U}_n$-small families of open sets of $X$. For each $\mathcal{U} \in \mathfrak{F}_n$ define $Z_n(\mathcal{U}) = \{f\in C_p(X, 2) : f(X\setminus \cup \mathcal{U}) = \{0\}\}$. For each $n$ choose $Y_n = \cup \{Z_n(\mathcal{U}) : \mathcal{U}\in \mathfrak{F}_n\}$. Observe that for each $n$ $Y_n$ is dense in $C_p(X, 2)$. Since $C_p(X, 2)$ is $\mathcal{I}$-H-separable, there exists a sequence $(F_n)$ such that for each $n$ $F_n$ is a finite subset of $Y_n$ and for every nonempty open set $U$ of $C_p(X, 2)$, $\{n\in \mathbb{N} : U\cap F_n = \emptyset\} \in \mathcal{I}$. For each $n$ and each $f\in F_n$ choose a $\mathcal{U}_f^{(n)} \in \mathfrak{F}_n$ such that $f\in Z_n(\mathcal{U}_f^{(n)})$. For every $V_1, V_2, \dotsc, V_k \in \mathcal{U}_f^{(n)}$ choose a $U_f^{(n)}(V_1, V_2, \dotsc, V_k) \in \mathcal{U}_n$ such that $V_1 \times V_2\times \dotsc V_k \subseteq U_f^{(n)}(V_1, V_2, \dotsc, V_k)$. For each $n$ $\mathcal{V}_n = \{U_f^{(n)}(V_1, V_2, \dotsc, V_k) : V_1, V_2, \dotsc, V_k \in \mathcal{U}_f^{(n)}, f\in F_n\}$ is a finite subset of $\mathcal{U}_n$. Observe that the sequence $(\mathcal{V}_n)$ witnesses for $(\mathcal{U}_n)$ that $X^k$ is $\mathcal{I}$-Hurewicz. Let $x = (x_1, x_2, \dotsc, x_k) \in X^k$. Pick the open set $U = \{f\in C_p(X, 2) : f(x_i) = 1 \text{ for all } 1\leq i \leq k\}$ of $C_p(X, 2)$. Clearly $U$ is nonempty. Then $\{n\in \mathbb{N} : U\cap F_n = \emptyset\} \in \mathcal{I}$. We claim that $\{n\in \mathbb{N} : x\notin \cup \mathcal{V}_n\} \subseteq \{n\in \mathbb{N} : U\cap F_n = \emptyset\}$. Pick a $m \in \{n\in \mathbb{N} : x\notin \cup \mathcal{V}_n\}$. Then $x \notin \cup \mathcal{V}_m$. If possible suppose that $U\cap F_m \neq \emptyset$. Pick a $f \in C_p(X, 2)$ such that $f\in U\cap F_m$. Since $f\in Z_m(\mathcal{U}_f^{(m)})$, $f(X\setminus \cup \mathcal{U}_f^{(m)}) = \{0\}$. It follows that $f(x_i) = 1$ for all $1\leq i \leq k$ and $f(y) = 0$ for all $y\in X\setminus \cup \mathcal{U}_f^{(m)}$. For each $1\leq i \leq k$ choose a $V_i \in \mathcal{U}_f^{(m)}$ such that $x_i \in V_i$. So $x \in V_1\times V_2 \times \dotsc V_k$. Then $x\in U_f^{(m)}(V_1, V_2, \dotsc, V_k)$ as $V_1 \times V_2\times \dotsc V_k \subseteq U_f^{(m)}(V_1, V_2, \dotsc, V_k)$ and hence $x\in \cup \mathcal{V}_m$. Which is absurd. Consequently $U\cap F_m \neq \emptyset$, i.e. $m\in \{n\in \mathbb{N} : U\cap F_n = \emptyset\}$ and so $\{n\in \mathbb{N} : x\notin \cup \mathcal{V}_n\} \subseteq \{n\in \mathbb{N} : U\cap F_n = \emptyset\}$, i.e. $\{n\in \mathbb{N} : x\notin \cup \mathcal{V}_n\} \in \mathcal{I}$. Thus $X^k$ is $\mathcal{I}$-Hurewicz.

$(4) \Rightarrow (1)$. By Theorem~\ref{thm5}, $C_p(X)$ is $\mathcal{I}$-H-separable. Since $X$ is zero-dimensional, $C_p(X, \mathbb{Q})$ is dense in $C_p(X)$. By Lemma~\ref{lemma103}, $C_p(X, \mathbb{Q})$ is $\mathcal{I}$-H-separable.
\end{proof}

\begin{Cor}
\label{cor7}
Let $(X, \tau)$ be condensed onto a separable metrizable space $(X, \tau^\prime)$. If every finite power of $(X, \tau^\prime)$ is $\mathcal{I}$-Hurewicz, then there exists a dense $\mathcal{I}$-H-separable subspace of $C_p(X, \tau^\prime)$.
\end{Cor}
\begin{proof}
We can think of $C_p(X, \tau^\prime)$ as subspace of $C_p(X, \tau)$. Clearly $C_p(X, \tau^\prime)$ is dense in $C_p(X, \tau)$. By Theorem~\ref{thm3}, $C_p(X, \tau^\prime)$ is $\iwfsd$. Since $(X, \tau^\prime)$ is separable metrizable, $C_p(X, \tau^\prime)$ is hereditarily separable. It follows that $C_p(X, \tau^\prime)$ is $\mathcal{I}$-H-separable (see Proposition~\ref{prop1}).
\end{proof}

We now discuss the importance of separability in the preceding result.
\begin{Ex}
\label{ex4}
There is a metrizable space $X$ such that $C_p(X)$ is not $\mathcal{I}$-H-separable but $C_p(X)$ has a dense $\mathcal{I}$-H-separable subspace.
\end{Ex}
\begin{proof}
Consider the discrete space $X$ with $|X| = \mathfrak{c}$. Now $C_p(X) = \mathbb{R}^\mathfrak{c}$. We can interpret $\mathbb{R}^\mathfrak{c}$ as $\mathbb{R}^{2^\omega}$. If $Y = C_p(2^\omega)$, then $Y$ is dense in $\mathbb{R}^{2^\omega}$. Now $iw(2^\omega) = \omega$ and $2^\omega$ is compact, i.e. every finite power of $2^\omega$ is $\mathcal{I}$-Hurewicz. By Theorem~\ref{thm5}, $Y$ is $\mathcal{I}$-H-separable. Thus $Y$ is a dense $\mathcal{I}$-H-separable subspace of $C_p(X) = \mathbb{R}^\mathfrak{c}$. Again by Theorem~\ref{thm5}, $C_p(X) = \mathbb{R}^\mathfrak{c}$ is not $\mathcal{I}$-H-separable as $X$ is not $\mathcal{I}$-Hurewicz.
\end{proof}

The following question can naturally be raised.
\begin{Prob}
\label{prob3}
Suppose that $C_p(X)$ has a dense $\mathcal{I}$-H-separable subspace. Can $X$ be condensed onto a second countable space every finite power of which is $\mathcal{I}$-Hurewicz?
\end{Prob}

\end{document}